\documentclass[reqno,12pt,letterpaper]{amsart}
\usepackage{amsmath,amssymb,amsthm,graphicx,mathrsfs,url}
\usepackage[usenames,dvipsnames]{color}
\usepackage[colorlinks=true,linkcolor=Red,citecolor=Green]{hyperref}
\usepackage{amsxtra}

\def\?[#1]{\textbf{[#1]}\marginpar{\Large{\textbf{??}}}}

\def\smallsection#1{\smallskip\noindent\textbf{#1}.}
\let\epsilon=\varepsilon 

\newcommand{\RR}{{\mathbb R}}

\newcommand{\NN}{{\mathbb N}}
\newcommand{\CC}{{\mathbb C}}
\newcommand{\HH}{{\mathbb H}}
\newcommand{\ZZ}{{\mathbb Z}}
\newcommand{\SP}{{\mathbb S}}
\newcommand{\CI}{{{\mathcal C}^\infty}}
\newcommand{\bCI}{{\bar{\mathcal C}^\infty}}
\newcommand{\dCI}{{\dot{\mathcal C}^\infty}}
\newcommand{\CIc}{{{\mathcal C}^\infty_{\rm{c}}}}

\newtheorem{thm}{Theorem}
\newtheorem{prop}{Proposition}

\newtheorem{lem}[prop]{Lemma}

\numberwithin{equation}{section}

\DeclareMathOperator{\Spec}{Spec}

\DeclareMathOperator{\Char}{Char}
\DeclareMathOperator{\Diff}{Diff}

\let\Im=\Imag

\DeclareMathOperator{\coker}{coker}

\let\Re=\Real

\DeclareMathOperator{\supp}{supp}
\DeclareMathOperator{\vol}{vol}
\DeclareMathOperator{\WF}{WF}

\title[Vasy's method revisited]{Resonances for asymptotically hyperbolic manifolds: Vasy's method
revisited}
\author{Maciej Zworski}
\email{zworski@math.berkeley.edu}
\address{Department of Mathematics, University of California,
Berkeley, CA 94720, USA}

\begin{document}

\begin{abstract}
We revisit Vasy's method \cite{vasy1},\cite{vasy2} for showing meromorphy of the resolvent for (even) asymptotically hyperbolic manifolds. It
provides an effective definition of resonances in that setting
by identifying them with poles of inverses of a family of Fredholm
differential operators. In the Euclidean case the method of 
complex scaling made this available since the 70's but 
in the hyperbolic case an effective definition was not known till 
\cite{vasy1},\cite{vasy2}. Here we present a simplified version which 
relies only on standard pseudodifferential techniques 
and estimates for hyperbolic operators. As a byproduct
we obtain more natural invertibility properties of the Fredholm family.

\end{abstract}

\maketitle

\section{Introduction}

We present a version of the method introduced by 
Andr\'as Vasy  \cite{vasy1},\cite{vasy2} to prove meromorphic continuations of resolvents of
Laplacians on even asymptotically hyperbolic spaces -- see 
\eqref{eq:gash}. That meromorphy was first established
for any asymptotically hyperbolic metric by Mazzeo--Melrose \cite{mm}.
Other early contributions were made by Agmon \cite{Ag}, Fay \cite{Fa}, Guillop\'e--Zworski \cite{GuZw}, Lax--Phillips \cite{LaPh}, Mandouvalos \cite{Man}, Patterson \cite{Pa} and Perry \cite{Pe}. 
Guillarmou \cite{g} showed that the evenness condition was needed for 
a global meromorphic continuation and clarified the construction given in 
\cite{mm}. 

Vasy's method is dramatically different from earlier approaches and is related to the
study of stationary wave equations for Kerr--de Sitter black holes -- see 
\cite{vasy1} and \cite[\S 5.7]{res}. Its advantage lies in relating
the resolvent to the inverse of a family of {\em Fredholm differential 
operators}. Hence, microlocal methods can be used to prove 
results which have not been available before, for instance
existence of resonance
free strips for non-trapping metrics \cite{vasy2}. 
Another application is the work 
of Datchev--Dyatlov \cite{DaDy} on the fractal upper
bounds on the number of resonances for (even) asymptotically hyperbolic 
manifolds and in particular for convex co-compact quotients of $ \HH^n $.
Previously only the case of convex co-compact Schottky quotients was known \cite{glz} and that was established using
transfer operators and zeta function methods. In the context of 
black holes the construction has been used to obtain 
a quantitative version of Hawking radiation \cite{Dr}, 
exponential decay of waves in the Kerr--de Sitter case \cite{Dy1}, 
the description of quasi-normal modes for perturbations of Kerr--de Sitter 
black holes \cite{Dy2}
and rigorous definition of quasi-normal modes for Kerr--Anti de Sitter 
black holes \cite{ga}. The construction of the Fredholm family 
also plays a role in the study of linear and
non-linear scattering problems -- see \cite{BaVWu}, \cite{HiV1}, \cite{HiV2}
and references given there. 

A related approach to meromorphic continuation, motivated by the study of 
Anti-de Sitter black holes, was independently developed by Warnick \cite{Wa}.
It is based on physical space techniques for hyperbolic equations
and it also provides meromorphic continuation of resolvents for even 
asymptotically hyperbolic metrics \cite[\S 7.5]{Wa}.

We should point out that for a large class of asymptotically Euclidean
manifolds an effective characterization of resonances has been known since the 
introduction of the method of complex scaling by Aguilar--Combes, 
Balslev--Combes and Simon in the 1970s -- see \cite[\S 4.5]{res} for an 
elementary introduction and references and \cite{WuZw} for a class
asymptotically Euclidean manifolds to which the method applies.

In this note we present a direct proof of meromorphic continuation 
based on standard pseudodifferential techniques and estimates for
hyperbolic equations which can found, for instance, in 
\cite[\S 18.1]{ho3}  and \cite[\S 23.2]{ho3} respectively.
In particular, we prove Melrose's {\em radial estimates} \cite{mel}
which are crucial for establishing the Fredholm property.
A semiclassical version of the approach presented here can be found in 
\cite[Chapter 5]{res} -- it is needed for the high energy results 
\cite{DaDy}, \cite{vasy2} mentioned above. 

\renewcommand\thefootnote{\dag}%

We now define {even asymptotically hyperbolic manifolds}. 
Suppose that $  \overline M$ is a compact manifold with boundary $ \partial M \neq \emptyset$ of dimension $ n +1 $. We denote by $ M $ the interior of $ \overline  M $. 
The Riemannian manifold $ ( M , g ) $ is even asymptotically hyperbolic
if there exits functions $ y' \in \bCI ( M ; \partial M )$ and
 $ y_1  \in {\bCI} (  M ; (0, 2 )  ) $\footnote{We cannot
write a paper about Vasy's method without some footnotes: we follow the notation 
of \cite[Appendix B]{ho3} where $ \bCI ( M )$ denotes functions which 
are smoothly extendable across $ \partial M $ and $ \dCI (  \overline M ) $ functions which are 
extendable to smooth functions {\em supported} in $ \overline M $ -- see
\S \ref{Fredamp}.}, 
$ y_1|_{\partial M } = 0 $, $ dy_1|_{\partial 
M } \neq 0 $, such that 
\begin{equation}
\label{eq:coords}  \overline M \supset y_1^{-1} ( [0 , 1 ] ) \ni m \mapsto ( y_1( m), y'( m ) ) 
\in [ 0 , 1 ] \times \partial M \end{equation}
is a diffeomorphism, 
and near $ \partial M $ the metric has the form,
\begin{equation}
\label{eq:gash}   g|_{ y_1 \leq 1}  = \frac{ dy_1^2 + h ( y_1^2 )  }{ y_1^2 } ,  \end{equation}
where 
 $ [ 0, 1 ] \ni t \mapsto h ( t ) $, 
is a smooth family of Riemannian metrics on $ \partial M $.
 For 
the discussion of invariance of this definition and of its geometric meaning
we refer to \cite[\S 2]{g}.

Let $ - \Delta_g \geq 0 $ be the Laplace--Beltrami operator for the metric $ g $.
Since the spectrum is contained in $ [ 0 , \infty ) $ the operator 
$- \Delta_g -\zeta ( n -\zeta ) $ is invertible on $ H^2 ( M , d
\vol_g ) $ for $ \Re\zeta > n $. Hence we can define
\begin{equation}
\label{eq:Res4Deltag}   R (\zeta) := ( - \Delta_g -\zeta ( n  - \zeta  ) )^{-1} : L^2 ( M, d\!\vol_g ) \to H^2 ( M, d \!\vol_g ) , \ \ 
\Re\zeta  > n . 
\end{equation}
We note that elliptic regularity shows that $ R ( \zeta ) : \dCI ( M ) \to 
\CI ( M )$, $ \Re\zeta  > n $. 
We also remark that as a byproduct 
of the construction we will show the well known fact that $ R ( \lambda ) 
: L^2 \to H^2 $ is meromorphic for $ \Re\zeta  > n/2 $: the poles correspond
to $ L^2 $ eigenvalues of $ - \Delta_g $ and hence lie in $ ( n/2, n ) $.

We will prove the result of Mazzeo--Melrose \cite{mm} and Guillarmou \cite{g}:
\begin{thm}
\label{t:1}
Suppose that $ ( M , g)  $ is an even asymptotically hyperbolic manifold
and that $ R ( \zeta ) $ is defined by \eqref{eq:Res4Deltag}.
Then 
\[    R (\zeta ) : \dCI ( M ) \to \CI ( M ) , \]
continues meromorphically from $ \Re\zeta > n $
to $ \CC$ with poles of finite rank.
\end{thm}

The key point however is the fact that $ R (\zeta) $ can be 
related to $ P( i ( \zeta - n/2)  )^{-1}$ where 
\[ \zeta \longmapsto  P ( i ( \zeta - n/2)  ) \]
 is a family of 
Fredholm differential operators -- see \S \ref{ffdo} and Theorem \ref{t:2}.
That family will be shown to be invertible for $ \Re\zeta > n $
which proves the meromorphy of $ P ( i ( \zeta - n/2) )^{-1} $ -- see Theorem \ref{t:3}. We remark that for $ \Re \zeta > \frac n 2 $, 
$ R ( \zeta )$ is meromorphic as an operator $ L^2 ( M ) \to L^2 ( M ) $ with 
poles corresponding to eigenvalues of $ - \Delta_g $.

The paper is organized as follows. In \S \ref{ffdo} we define the
family $ P ( \lambda ) $ and the spaces on which it has the Fredholm property. 
That section contains the main results of the paper: Theorems \ref{t:2}
and \ref{t:3}. In \S \ref{Fredamp} we recall the notation from the
theory of pseudodifferential operators and provide detailed references. 
We also recall estimates for hyperbolic operators needed here.
In \S \ref{radest} we prove Melrose's propagation estimates at radial 
points and in \S \ref{s:t1} we use them to show the Fredholm property. 
\S\ref{asym} gives some precise estimates valid for $ \Im \lambda 
\gg 1 $. Finally \S \ref{merc} we present invertibility of $ P ( \lambda ) $
for $ \Im \lambda \gg 1 $ and that proves the meromorphic continuation.
Except for references to \cite[18.1]{ho3} and \cite[23.2]{ho3} and 
some references to standard approximation arguments \cite[Appendix E]{res} (with 
material readily available in many other places) the paper is self-contained.

\smallsection{Acknowledgements}
I would like to thank Semyon Dyatlov and Andr\'as Vasy for helpful
comments on the first version of this note. 
I am particularly grateful to Peter Hintz for many suggestions and
for his help with the proof of Proposition \ref{p:hintz}. 
Partial support by the National Science Foundation 
under the grant DMS-1500852 is also gratefully acknowledged.

\section{The Fredholm family of differential operators}
\label{ffdo}

Let $ y' \in \partial M $ denote the variable on $ \partial M $.
Then \eqref{eq:gash} implies that near $ \partial M $, the Laplacian has the form 
\begin{gather}
\label{eq:Deltagg} 
\begin{gathered} - \Delta_g = ( y_1 D_{y_1} )^2 + i ( n  + y_1^2 \gamma ( y_1^2, y') ) y_1
D_{y_1} - y_1^2 \Delta_{h (y_1^2) } , \\
\gamma ( t, y') := - \partial_t \bar h ( t ) / \bar h ( t ) , \ \
\bar h ( t ) := \det h ( t ) . 
\end{gathered}
\end{gather}
Here  $ \Delta_{h (y_1^2 ) } $ is the Laplacian for the family of metrics on $ \partial M $ depending smoothly on $ y_1^2 $ and 
$ \gamma \in \CI ( [0,1]\times \partial M ) $. (The logarithmic derivative 
defining $ \gamma $ is independent of of the density on $ \partial M$ needed to 
define the determinant $ \bar h $.)

In \S \ref{asym} we will show that the unique $ L^2 $ solutions to 
\[ ( - \Delta_g -\zeta ( n -\zeta )  ) u =f \in 
\dCI ( M) , \ \ \Re\zeta > n,  , \ \ 
\]
satisfy
\[  u = y_1^{\zeta } \bCI ( M )  \ \ \ \text{ 
and $ \ \ \ \   y_1^{ -\zeta }u |_{ y_1 < 1 } = 
F ( y_1^2 , y') , \ \ F \in \bCI ( [0,1 ]\times \partial M )  $.} \]
Eventually we will show that the meromorphic continuation of the resolvent 
provides solutions of this form for all $\zeta \in \CC $ that are not poles
of the resolvent.

This suggests two things:

\begin{itemize}

\item To reduce the investigation to the study of smooth solutions we 
should conjugate $ - \Delta_g -\zeta ( n- \zeta) $ by the weight $ y_1^\zeta $.

\item The desired smoothness properties should be stronger in the 
sense that the functions should be smooth in $ (y_1^2 , y' ) $.

\end{itemize}

Motivated by this we calculate, 
\begin{equation}
\label{eq:firstconj} y_1^{ -\zeta}   ( - \Delta_g - \zeta ( n - \zeta)  ) y_1^{\zeta } =
x_1 P ( \lambda ) , \ \   x_1 = y_1^2 , \ \ x' = y', \ \lambda = i (\zeta - {\textstyle{\frac n 2 }})  , \end{equation}
where, near $ \partial M $, 
\begin{equation}
\label{eq:Plag} P ( \lambda ) = 4 ( x_1 D_{x_1}^2 - ( \lambda + i ) D_{x_1} ) 
- \Delta_h + i \gamma ( x ) \left( 2 x_1 D_{x_1} - \lambda - i 
{\textstyle\frac{ n-1} 2 } \right) . \end{equation}
The switch to $ \lambda $ is motivated by the fact that numerology 
is slightly lighter on the $ \zeta$-side for $ - \Delta_g $ and on the $\lambda$-side
for $ P ( \lambda ) $.

\renewcommand\thefootnote{\ddag}%

To define the operator $ P ( \lambda ) $ geometrically we introduce a new
manifold using coordinates \eqref{eq:coords} and $ x_1 = y_1^2 $ for $ y_1 > 0 $:
\begin{equation}
\label{eq:coordX}  X = [ -1 , 1 ]_{x_1}  \times \partial M \sqcup \left( M \setminus y^{-1} 
(( 0, 1 ) ) \right). \end{equation}
We note that $ X_1 := X \cap \{ x_1 > 0 \} $ is diffeomorphic to $ M $ but
$ \overline X_1 $ and $ \overline M $ have different 
$ \CI $-structures\footnote{This construction appeared already in \cite[\S 2]{GuZw}
and $ P ( \lambda ) = Q ( n/4 - i \lambda/2) $ where $ Q ( \zeta ) $ was
defined in \cite[(2.6),(3.12)]{GuZw}. However 
the significance of $ Q ( \zeta ) $ did not become clear until \cite{vasy1}.}.

We can extend $ x_1 \to h ( x_1 ) $ to a family of smooth non-degenerate
metrics on $ \partial M $ on $ [-1,1]$. Using \eqref{eq:Deltagg} 
that provides a natural extension 
of the function $ \gamma $ appearing  \eqref{eq:firstconj}.

The Laplacian $ -\Delta_g $ is a self-adjoint operator on
$ L^2 ( M , d \vol_g ) $, where near $ \partial M $ and in the notation of 
\eqref{eq:Deltagg}, 
\[  d \vol_g =  y_1^{-n-1} \bar h ( y_1^2, y') dy_1 dy' , \]
where $d y' $ in a density on $ \partial M $ used to define the 
determinant $ \bar h = \det h $. The conjugation \eqref{eq:firstconj} shows
that for $ \lambda \in \RR $ ($ \zeta \in \frac n2 + i \RR $) 
$ x_1 P ( \lambda ) $ is formally self-adjoint with respect to 
$ x_1^{-1} \bar h ( x) dx_1 dx' $ and consequently $ P ( \lambda ) $
is formally self-adjoint for 
\begin{equation}
\label{eq:dmug}  d \mu_g =  \bar h ( x) dx . \end{equation}
This will be the measure used for defining $ L^2  ( X) $ in what follows.
In particular we see that the formal adjoint with respect to $ d \mu_g $ satisfies
\begin{equation}
\label{eq:formaladj}
P ( \lambda )^* = P ( \bar \lambda ) .
\end{equation}

We can now define spaces on which $ P ( \lambda ) $ is a Fredholm operator.
For that we denote by $ \bar H^s ( X^\circ ) $ the space
of restrictions of elements of $ H^s $ on an extension of $ X$ across the 
boundary to the interior of $ X $ -- see \cite[\S B.2]{ho3} and \S \ref{hypest} -- and put
\begin{equation}
\label{eq:hypXY}  \mathscr Y_s := \bar H^s ( X^\circ ) , \ \
\mathscr X_s := \{ u \in \mathscr Y_{s+1} : P ( 0 ) u \in \mathscr Y_s \}. 
\end{equation}
Since the dependence on $ \lambda $ in $ P ( \lambda ) $ occurs only in 
lower order terms we can replace $ P ( 0 ) $ by $ P ( \lambda ) $ in the
definition of $ \mathscr X $.


\noindent
{\bf Motivation:} Since for $ x_1 < 0 $ the operator $ P ( \lambda ) $
is hyperbolic with respect to surfaces $ x_1 = a > 0 $ the following
elementary example motivates the definition \eqref{eq:hypXY}.  Consider
$ P = D_{x_1}^2 - D_{x_2}^2 $ on $ [ -1, 0 ] \times \SP^1 $ and define
\[  Y_s := \{ u \in \bar H^s ( [-1, \infty ) \times \SP^1 ) : \supp u \subset 
[ -1, 0 ] \times \SP^1 \}, \ \ X_s := \{ u \in Y_{s+1} : P u \in Y_s \}. \]
Then standard hyperbolic estimates -- see for instance 
\cite[Theorem 23.2.4]{ho3} -- show that for any $ s \in \RR $, the operator
$ P : X_s \to Y_s $ is invertible. Roughly, the support condition 
gives $ 0 $ initial values at $ x_1 = 0 $ and hence $ P u = f $ can 
be uniquely solved for $ x_1 < 0 $. 


We can now state the main theorems of this note:
\begin{thm}
\label{t:2}
Let $ \mathscr X_s , \mathscr Y_s $ be defined in \eqref{eq:hypXY}.
Then for $ \Im \lambda > - s - \frac12 $ the operator
\[   P ( \lambda ) : \mathscr X_s \to \mathscr Y_s , \]
has the Fredholm property, that is 
\[  \dim \{ u \in \mathscr X_s : P ( \lambda ) u = 0 \} < \infty , \ \
\dim \mathscr Y_s / P ( \lambda ) \mathscr X_s < \infty , \]
and $ P ( \lambda ) \mathscr X_s $ is closed.
\end{thm}

The next theorem provides invertibility of $ P ( \lambda ) $ for
$ \Im \lambda > 0  $ and that shows the meromorphy of $ P ( \lambda )^{-1} $
-- see \cite[Theorem C.4]{res}. We will use that in Proposition 
\ref{p:hintz} to show the well known fact that
in addition to Theorem \ref{t:1} $ R ( \frac n 2 - i \lambda ) $ is meromorphic
on $ L^2 ( M , d \! \vol_g ) $ for $ \Im \lambda > 0 $. 

\begin{thm}
\label{t:3}
For $ \Im \lambda > 0 $, $ \lambda^2 + (\frac n 2 )^2 \notin 
\Spec ( - \Delta_g ) $ and $ s > -\Im \lambda - \frac12 $, 
\[ P ( \lambda ) : \mathscr X_s \to \mathscr Y_s \]
is invertible. Hence,  for $ s \in \RR $ and $ \Im \lambda > - s - \frac12 $,
$ \lambda \mapsto P ( \lambda )^{-1} : \mathscr Y_s \to \mathscr X_s , $
is a meromorphic family of operators with poles of finite rank.
\end{thm}

For interesting applications it is crucial to consider the
semiclassical case, that is, uniform analysis as $ \Re \lambda \to 
\infty $ -- see \cite[Chapter 5]{res} -- 
 but to indicate the basic mechanism behind the meromorphic 
continuation we only present the Fredholm property and invertibility 
in the upper half-plane.

\section{Preliminaries}
\label{Fredamp}

Here we review the notation and basic facts need in the 
proofs of Theorems \ref{t:2} and \ref{t:3}. 

\subsection{Pseudodifferential operators}
\label{pseudo}

We use the notation of \cite[\S 18.1]{ho3} and 
for $ X $,  an open $ \CI $-manifold 
we denote by $ \Psi^m ( X ) $ the space of {\em properly supported}
pseudodifferential operators of order  $m $. (The operator $ A : 
\CIc ( X ) \to \mathcal D' ( X) $ is properly supported if
the projections from support of the Schwartz kernel of $ A $ in $ X \times X $ 
to each factor are proper maps, that is inverse images of compact sets
are compact. The support of the Schwartz kernel of any differential
operator is contained in the diagonal in $ X \times X $ and clearly 
has that property.) 

For $ A \in \Psi^m (  X) $ we denote by 
$ \sigma ( A ) \in 
S^{m}  ( T^*X \setminus 0 ) / S^{m-1} ( T^*X \setminus 0 ) $
the symbol of $ A $, sometimes writing $ \sigma ( A ) = a \in 
S^m ( T^* X \setminus 0 ) $ with an understanding that $ a $ is a
{\em representative} from the equivalence class in the quotient.

We will use the following basic properties of the symbol map:
if $ A \in \Psi^m ( X ) $ and $ B \in \Psi^k ( X ) $ then 
\begin{gather*} \sigma ( A  B ) = \sigma ( A )\sigma (  B) \in S^{m+k}/ S^{m+k-1} 
, \\
\sigma ( [ A, B ] ) = H_{ \sigma ( A ) } \sigma ( B ) \in 
S^{m+k -1} / S^{m+k-2} , \end{gather*}
where for $ a \in S^m $,  $ H_a $ is the {\em Hamiton vector field} of $ a $.

For any operator $ P \in \Psi^m ( X) $ we 
can define $ \WF ( P ) \subset T^* X \setminus 0 $ 
(the smallest subset outside of which $ A $ has order $ - \infty $
-- see \cite[(18.1.34)]{ho3}). We also define $ \Char ( P) $ the 
smallest conic closed set outside of which $ P $ is {\em elliptic} -- 
see \cite[Definition 18.1.25]{ho3}. 
A typical application of the symbolic calculus and of this notation is
the following statement \cite[Theorem 18.1.24$'$]{ho3}: if $ P \in \Psi^m ( X ) $
and $ V $ is an open conic set such that $ V \cap \Char ( P) = \emptyset $
then there exists $ Q \in \Psi^{-m} ( X ) $ such that 
\begin{equation} 
\label{eq:WFPQ}    \WF ( I - P Q ) \cap V = \WF ( I - Q P  ) \cap V = \emptyset . 
\end{equation}
This means that $ Q $ is a {\em microlocal} inverse of $ P $ in $ V $.

We also recall that the operators in $ A \in \Psi^m ( X ) $ have mapping properties
\[  A : H^s_{\rm{loc} } ( X ) \to H^{s-m}_{\rm{loc}} ( X ) , \ \
 A : H^s_{\rm{comp} } ( X ) \to H^{s-m}_{\rm{comp}} ( X ) , \ \ 
 s \in \RR .\]
Combined with \eqref{eq:WFPQ} we obtain the following {\em elliptic} 
estimate: if $ A, B \in \Psi^0 ( X ) $ have 
{\em compactly supported} Schwartz kernels, 
 $ P \in \Psi^m ( X )$ and 
\[ \WF ( A ) \cap ( \Char ( B ) \cup \Char ( P ) ) = \emptyset , \]
then for any $N $ there exists $ C $ such that
\begin{equation}
\label{eq:elle4P}
\| A u \|_{ H^{s+m} } \leq C \| B P u \|_{H^s} + C \| u \|_{ H^{-N}}. 
\end{equation}

\subsection{Hyperbolic estimates}
\label{hypest} 

If $ X $ is a smooth compact manifold with boundary we follow \cite[\S B.2]{ho3} 
and define Sobolev spaces of extendible distributions, $ \bar H^s ( X^\circ ) $ 
and of supported distributions $ \dot H^s ( X ) $. Here $ X = X^\circ \sqcup 
\partial X $ and $ X^\circ $ is the interior of $ X $. These are modeled
on the case 
of $ X = \overline \RR^{n}_{+} $, $ \RR_+^n := \{ x \in \RR^n : x_1 > 0 \}$
in which case
\begin{gather*}
  \bar H^s( \RR_+^n ) = \{ u : \exists \, U \in H^s ( \RR^n ) , \  u= U|_{ 
x_1 > 0 }\}  , \\ \dot H^s ( \overline \RR^n_+ ) := \{ 
u \in H^s ( \RR^n ) : \supp u \subset \overline \RR^n_+ \}. \end{gather*}

The key fact is that the $ L^2 $ pairing (defined using a smooth density on $ X $)
\[   \dCI ( X ) \times  \bCI ( X^\circ ) \ni ( u , v ) \mapsto 
\int_X u ( x ) \bar v ( x ) dx , \]
extends by density to $ ( u , v ) \in \dot H^{-s} ( X ) \times \bar H ( X^\circ ) $
and provides the identification of dual spaces, 
\begin{equation}
\label{eq:duality} ( \bar H^s ( X^\circ ) )^* \simeq \dot H^{-s} ( X ) , \ \ s \in \RR . \end{equation}

Suppose that $ P = D_t^2 + P_1 ( t, x, D_x ) D_t + P_0 ( t, x , D_x ) $, $ x \in N $, where
$ N $ is a compact manifold and $ P_j  \in \CI ( \RR_t ; \Psi^{2-j} ( N ) ) $
is strictly hyperbolic with respect to the level surfaces $ t = \rm{const} $ -- 
see \cite[\S 23.2]{ho3}. 
For any 
$ T > 0 $ and $ s \in \RR $, we define
\[   \widetilde H^s ( [0 , T ) \times N) = 
\left\{ u : u = U|_{ [ 0 , T ) \times N } , \ \ U \in 
 H^s ( \RR \times N ), 
\ \supp U \subset [ 0 , \infty )  \times N \right\}, \]
with the norm defined as infimum of $ H^s $ norms over all $ U \in H^s $
with $ u_{ [ 0 , T) } = U $. (These spaces combines the $ \dot H^s $ space
at the $ t = 0 $ with $ \bar H^s $ at $ t = T $.)

Then
\begin{equation}
\label{eq:hypeP}
\forall \, f \in \widetilde H^s ( [ 0 , T ) \times N ) \, \ 
\exists \, ! \, u \in \widetilde H^{s+1} ( [0 , T ) \times N ) , \ \ 
P u = f , 
\end{equation}
and 
\begin{equation}
\label{eq:hypePest}
\| u \|_{ \widetilde H^{s+1} ( [ 0 , T ) \times N ) } 
\leq C \| f \|_{\widetilde H^s ( [ 0 , T ) \times N ) },
\end{equation}
see \cite[Theorem 23.2.4]{ho3}. 

If we define
\[  Y_s := \widetilde H^s ( [ 0 , T ) \times N ) , \ \
X_s := \{ u \in Y_{s+1} : P u \in Y_s \} \]
then  $ P : X_s \to Y_s $ is {\em invertible}. 
In our application we will need the following estimate which
follows from the invertibility of $ P $:
if $ u \in \bar H^s ( ( 0 , T ) \times N ) $ then for any $ \delta > 0 $, 
\begin{equation}
\label{eq:hypest1}
\| u \|_{ \bar  H^{s+1}  ( ( 0 , T )\times N ) } \leq C 
\| P u \|_{ \bar H^s ( ( 0 , T )  \times N ) } + C \| u \|_{ 
\bar H^{s+1}  ( ( 0 , \delta  ) ) \times N ) } .
\end{equation}

The operator $ P ( \lambda ) $ defined in \eqref{eq:Plag} is of the form 
$ x_1 ( D_{x_1}^2 - P_1 ( x ) D_{x_1} + P_0 ( x , D_{x'} ) ) $ where $ P_1 \in \CI $ and $ P_0  $ is 
elliptic for $ -1 \geq x_1 < - \epsilon < 0 $, for any fixed $ \epsilon$.
That means that for $ t = 1 + x_1 $ and $ T = 1 - \epsilon $ or
$ t = - \epsilon - x_1 $, $ T = 1 - \epsilon $, the operator is (up to the
non-zero smooth factor $ x_1 $) is of the form to which estimates
\eqref{eq:hypePest} and \eqref{eq:hypest1} apply.

We will also need an estimate valid all the way to $ x_1 = 0 $:

\begin{lem}
\label{l:ahe}
Suppose that 
$  u \in \dCI ( X \cap \{ x_1 \leq 0 \} ) $ and 
$ P ( \lambda ) u = 0 $. 
Then $ u \equiv 0 $.
\end{lem}
As pointed out by Andr\'as Vasy 
this follows from general 
properties of the de Sitter wave equation \cite[Proposition 5.3]{vasy3}
but we provide a simple direct proof.

\begin{proof}
We note that if $ u|_{ x_1 \geq - \epsilon } = 0 $ 
for some $ \epsilon > 0 $ then $ u \equiv 0 $  
by \eqref{eq:hypePest}. That follows from energy estimates. We want to make that
argument quantitative.  We will work in $ [ -1, -\epsilon ] \times \partial M $
and define $ d : \CI ( \partial M ) \to \CI ( \partial M ; T^* \partial M ) $ 
to be the differential. We denote by $ d^* $ its Hodge adjoint with
with 
respect to the ($ x_1 $-dependend) metrics $ h $,
$d^*_h :  \CI ( \partial M ; T^* \partial M ) \to \CI ( \partial M ) $. Then
\[  P ( \lambda ) = 4 x_1 D_{x_1}^2 + d^*_h d - 4 ( \lambda + i ) D_{x_1} 
-  i \gamma ( x ) ( 2x_1 D_{x_1}  - \lambda - i {\textstyle{\frac{n-1}2}} ) . \]
Since for $ f \in \CI ( \partial M )  $ and any fixed $ x_1$, $ h = h ( x_1 ) $, 
\[ \begin{split} \int_{\partial M }  d_{h}^* ( v  d u  ) \bar f \, d \! \vol_h & = \int_{\partial M } \langle v d u , d f \rangle_h \,  d \! \vol_h 
= \int_{\partial M } 
\left( \langle du , d ( \bar v f ) \rangle_h - \langle d u , d \bar v \rangle_h \bar f\, 
\right) d\! \vol_h \\ 
& =  \int_{\partial M }  \left(  v d^*_h d u -  \langle d u , d \bar v \rangle \right) \bar f\,  d \!\vol_h ,
\end{split} \]
we conclude that 
 $ d_h^* ( v du ) = v d^*_h d u -  \langle d u , d \bar v \rangle_h $. 
From this we derive the 
following form of the energy identity valid for $ x_1 < 0 $:
\[ \begin{split}  &  
\partial_{x_1} \left( |x_1|^{-N} ( -x_1 |\partial_{x_1} u |^2 + 
| d u |_h^2 + | u |^2 ) \right) + 
|x_1|^{-N} d_h^* \left( \Re ( \bar u_{x_1} d u ) \right) = \\ 
& \ \ 2 \Re |x_1|^{-N} \bar u_{x_1}  P ( \lambda ) u  
-  N |x_1|^{-N-1} \left( - x_{1} | u_{x_1}|^2 + | du |_h^2 + |u|^2 \right) + 
|x_1|^{-N} R ( \lambda , u ) , \end{split}
\] 
where $ R ( \lambda, u ) $ is a quadratic form in $ u $ and $ du $, independent
of $ N $. 
We now fix $ \delta >  0 $ and apply Stokes theorem in $ [ - \delta, -\epsilon ]
\times M $. For $ N $ large enough (depending on $ \lambda $) 
that gives
\[ \begin{split}
\int_{\partial M  } ( | u_{x_1}|^2 + | d u |_h^2 )|_{x_1 = \delta } 
\, d \vol_h 
& \leq 
C \epsilon^{-N} \int_{\partial M  } 
( | u_{x_1}|^2 + | d u |_h^2 )|_{x_1 = \epsilon }  \, d \vol_h  \\
& \leq C_K \epsilon^{-N+K} , \end{split}\]
for any $ K $, as $ \epsilon \to 0 + $ (since $ u $ vanishes
to infinite order at $ x_1 = 0 $). By choosing $K > N $ we see
that the left hand side is $ 0 $ and that implies that $ u $ is
zero. 
\end{proof}

\section{Propagation of singularities at radial points}
\label{radest}

To obtain meromorphic continuation of the resolvent
\eqref{eq:Res4Deltag} we need propagation 
estimates at {\em radial} points. 
These estimates were developed by Melrose \cite{mel} in the 
context of scattering theory on asymptotically Euclidean spaces
and are crucial in the Vasy approach \cite{vasy1}. A semiclassical 
version valid for very general sinks and sources was given in 
Dyatlov--Zworski \cite{dz} (see also \cite[Appendix E]{res}).

To explain this estimates we first review the now 
standard results on propagation of singularities due to H\"ormander 
\cite{ho-pc}. Thus let 
$ P \in \Psi^m ( X ) $, with a real valued symbol $p := \sigma ( P ) $.
Suppose that in  an open conic subset of $ U \subset T^*X \setminus 0 $, $ \pi ( U ) \Subset X $ ($ \pi : T^*X \to X $),
\begin{equation}
\label{eq:condp}   
p ( x , \xi ) = 0 ,  
\ ( x, \xi ) \in U \ \Longrightarrow \ \text{ $ H_p  $ and $ \xi \partial_\xi  $ are linearly independent at $ ( x , \xi ) $.}
\end{equation}
Here $ H_p $ is the Hamilton vector field of $ p $ and $ \xi \partial_\xi $ 
is the {\em radial} vector field. The latter is invariantly defined 
as the generator of the $ \RR_+ $ action on $ T^* X \setminus 0 $
(multiplication of one forms by positive scalars). 

The basic propagation estimate is given as follows: suppose that
$ A, B , B_1  \in \Psi^0 ( X ) $ and
$  \WF ( A ) \cup \WF ( B ) \subset U $, $ \WF ( I - B_1 ) \cap U = \emptyset$.

We also assume that that $ \WF ( A ) $ is
{\em forward controlled} by $ \complement \Char ( B ) $ 
in the following sense: for any $ ( x, \xi ) \in \WF(A) $ there exists $ T > 0$ 
such that 
\begin{equation}
\label{eq:control4P}
\exp ( - T H_p ) ( x, \xi )  \notin \Char ( B ), \ \   
\exp ( [-T , 0 ] H_p ) ( x, \xi)  \subset U ,
\end{equation}
 The forward control can be replaced by backward control,
that is we can demand existence of $ T < 0 $. That is allowed since the symbol
is real.  

The crucial estimate is then given by 
\begin{equation}
\label{eq:DH}  \| A u \|_{ H^{s+m-1} } \leq C \| B_1 P u \|_{H^{s}} 
+ C \| B u \|_{ H^{s+m-1} } + C \| u \|_{ H^{-N} } , 
\end{equation}
where $ N $ is arbitrary and $ C $ is a constant depending on $ N $. 
A direct proof can be found in \cite{ho-pc}. The estimate is valid
with  $ u \in \mathcal D' ( X ) $ for which the right hand side
is finite -- see \cite[Exercise E.28]{res}. 

We will consider a situation in which the condition \eqref{eq:condp} is violated. We will work on the manifold $ X $ given by \eqref{eq:coordX},
near $ x_1 = 0 $.  In the notation of \eqref{eq:condp} we assume that, near $ x_1 = 0 $,
\begin{gather}
\label{eq:modred}   
\begin{gathered} 
P \in \Diff^2 ( X ) , \ \ p = \sigma ( P ) =  x_1 \xi_1^2 + q ( x, \xi' ) , \ \ 
q ( x_1 , x' , \xi') := 
| \xi' |_{ h ( x_1, x') }^2, 
\end{gathered}
\end{gather} 
$ ( x', \xi' ) \in T^* \partial M $, $( x, \xi ) \in T^* X \setminus 0 $. 
The Hamilton vector field is given by 
\begin{equation}
\label{eq:Hpmod}  H_p = \xi_1 ( 2 x_1 \partial_{x_1} - \xi_1 \partial_{\xi_1} ) + 
\partial_{x_1} q ( x, \xi') \partial_{ \xi_1 }  +  H_{q( x_1) }  , \end{equation}
where $ H_{ q ( x_1 ) } $ is the Hamilton vectorfield of 
$ ( x' , \xi' ) \mapsto q ( x_1, x' ,\xi') $ on $ T^* \partial M $. 

We see that the condition \eqref{eq:condp} is violated at 
\begin{gather}
\label{eq:radial1}
\begin{gathered}   \Gamma =  \{ ( 0 , x' , \xi_1 , 0 ) : x' \in \partial M , \xi_1 \in \RR  \setminus 0 \} \subset T^* X \setminus 0 ,   \\ 
\Gamma = N^*Y \setminus 0 , \ \ Y := \{ x_1 = 0 \} . 
\end{gathered}
\end{gather}
In fact, $ H_p |_{ N^* Y } = -\xi_1 ( \xi \partial_\xi |_{ N^* Y } ) $. 
Nevertheless Propositions \ref{p:rad1} and \ref{p:rad2} below provide
propagation estimates valid in spaces with restricted regularity. 

We note that $ \Gamma = p^{-1} ( 0 ) \cap \pi^{-1} ( Y )   $ and that near 
$ \pi^{-1} ( Y ) $, $ \Sigma=: p^{-1} ( 0  ) $ has two {\em disjoint} connected 
components:
\begin{gather}
\label{eq:Sigmapm}
\begin{gathered}  
\Sigma = \Sigma_+ \sqcup \Sigma_- , \ \ 
\ \ \Gamma_\pm := \Sigma_\pm \cap \Gamma , \\
\Sigma_\pm \cap 
\{ |x_1| < 1 \} := \{ (- q ( x, \xi')/\rho^2 , x', \rho , \xi' ) : \pm \rho > 0 , \
|x_1 | < 1  \} .
\end{gathered}
\end{gather}
The set $ \Gamma_+ $ is a source and $ \Gamma_- $ is a sink for the flow
projected to the sphere at infinity -- see Fig.~\ref{f:radial}.

\begin{figure}
\includegraphics[scale=1]{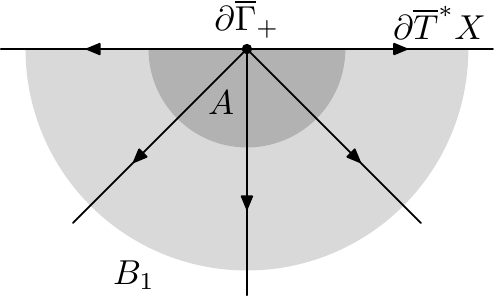}
\quad
\includegraphics[scale=1]{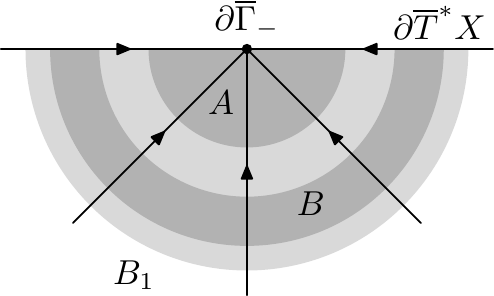}
\caption{An illustration of the behaviour of 
the Hamilton flows for radial sources and for radial sinks
and 
of the localization of operators in the estimates 
\eqref{eq:sourcest} and \eqref{eq:sinkest} respectively.
 The horizontal line
on the top denotes the boundary, $\partial\overline T^*X$,
of the {\em fiber--compactified} cotangent bundle $ \overline T^*X$. 
The shaded half-discs then correspond to conic neighbourhoods
in $ T^* X $. In the simplest example of $ X = ( - 1 , 1 ) \times \RR/\ZZ $, 
and $ p = x_1 \xi_1^2 + \xi_2^2 $, $ H_p = 
\xi_1 ( 2 x_1 \partial_{x_1} - \xi_1 \partial_{\xi_1} ) + 2 \xi_2 \partial_{x_2} $,
$ x_2 \in \RR/\ZZ $.
Near $ \partial \overline \Gamma_\pm $ explicit (projective) 
compactifications is given by $ r = 1/|\xi_1|$, (so that
$ \partial \overline T^* X = \{ r = 0 \} $), $ \theta = 
\xi_2/|\xi_1| $, with $ x $ (the base variable) unchanged.
In this variables, near $ \partial \overline \Gamma_\pm $ (boundaries of compactifications of $ \Gamma_\pm$ we check that 
$ r \partial_r = - \xi_1 \partial_{\xi_1 } - \xi_2 \partial_{\xi_2}$
and $  \theta \partial_\theta = \xi' \partial_{\xi'} $.
Hence near $ \Gamma_\pm $, 
$ H_p = \pm r (  \theta \partial_\theta + 
r \partial_r + 2 x_1 \partial_{x_1 }  + 2 \theta \partial_{x_2} )$ 
and (after rescaling) we see a source and a sink.}
\label{f:radial}
\end{figure}

We now write $ P $ as follows:
\begin{equation}
\label{eq:PQR}
P = P_0 + i Q , \ \   P_0 = P_0^* , \ \ Q = Q^* , 
\end{equation}
where the formal $ L^2$-adjoints are taken with respect to the density 
$ dx_1 d\vol_h $. 

We can now formulate the following propagation result at the source. 
We should stress that changing $ P $ to $ - P $ changes a source into 
a sink and the relevant thing is the sign of $ \sigma ( Q ) \in S^1/S^0 $
which then changes -- see \eqref{eq:s01} below.

We first state a {\em radial source estimate}:
\begin{prop}
\label{p:rad1}
In the notation of \eqref{eq:Sigmapm} and \eqref{eq:PQR} put 
\begin{equation}
\label{eq:s01}   s_+= \sup_{\Gamma_+}   |\xi_1|^{-1} \sigma ( Q ) - \textstyle{ \frac12} , \end{equation}
and take $ s > s_+ $.  For any $ B_1 \in \Psi^0 ( X ) $ satisfying
$ \WF ( I - B_1 )  \cap \Gamma_+ = \emptyset $ there exists $ A \in \Psi^0 ( X )  $ 
with $ \Char ( A ) \cap \Gamma_+ = \emptyset $ such that 
for $ u \in \CIc ( X ) $
\begin{equation}
\label{eq:sourcest}   \| A u \|_{ H^{s+1} } \leq C \| B_1 P u \|_{ H^s } + 
 C\| u \|_{H^{-N}} , \end{equation}
for any $ N $.
\end{prop}


\noindent
{\bf Remarks.} 1. The supremum in \eqref{eq:s01} should be understood
as being taken at the $ \xi$-infinity or as
$ s_+ = \sup_{ x' \in \partial M }  \lim_{\xi_1 \to \infty }  |\xi_1|^{-1} \sigma ( Q ) ( 0, x', \xi_1 , 0 )  - \frac12 $. 

\noindent
2.  An approximation argument -- see \cite[Lemma E.42]{res} -- 
shows that \eqref{eq:sourcest} is  
valid for $ u \in H^{-N} $, $ \supp u \cap X = \emptyset $, 
such that $ B_1 u \in H^{s+1} $, $ B_1 P u \in H^{s} $.

\noindent
3. Using a regularization argument -- see for instance 
\cite[\S 3.5]{ho-pc} or \cite[Exercises E.28, E.33]{res} -- 
\eqref{eq:sourcest} holds for all $ u \in \mathcal D' ( X )$, 
$ \supp u \subset K $ where $ K $ is a fixed compact subset of 
$ X^\circ $,  
such that $ B_1 u \in H^{r} $ for some $ r > s_+ +1 $.
In particular, when combined with the hyperbolic estimate
\eqref{eq:hypest1}, that gives 
\begin{equation}
\label{eq:PtoC} 
P u \in \bCI (  X) , \ \ u \in \bar H^r ( X ) , \ \ r > s_+ + 1 
\ \Longrightarrow \ u \in \bCI ( X ) . \end{equation} 
In fact, the smoothness near $ x_1 = 0 $ is obtained from 
the estimate \eqref{eq:sourcest} and elliptic 
estimates applied to $ \chi u $, $ \chi \in \CIc ( X ) $ and
then the hyperbolic estimates show smoothness for $ x_1 < - \epsilon $.

\noindent
4. To see that the threshold \eqref{eq:s01} is essentially optimal for 
\eqref{eq:PtoC}  
we consider $ X = ( - 1, 1 ) \times \RR/\ZZ $ and
$ P = x_1 D_{x_1}^2 - i ( \rho + 1 ) D_{x_1} - D_{x_2}^2 $, $ x_2 \in 
\RR/\ZZ $, $ \rho \in \RR $. In this case $ s_+ = -\rho - \frac12 $.
Put $  u ( x ) := \chi ( x_1 ) (x_1)_{+}^{-\rho} $, $  \rho \notin - \NN $, and
and note that 
\[ ( x_1 D_{x_1}^2 - i ( \rho + 1 ) D_{x_1 }) (x_1)_+^{-\rho} = 0 . \]
Hence $ P u \in \CIc ( X ) $ and $ u \in H^{-\rho + \frac12 - } \setminus
H^{ - \rho + \frac12} $. 


The {\em radial sink estimate} requires a control condition similar to that in 
\eqref{eq:control4P}. There is also a change in the regularity condition.

\begin{prop}
\label{p:rad2}
In the notation of \eqref{eq:Sigmapm} and \eqref{eq:PQR} put 
\begin{equation}
\label{eq:s02}  s_- = \sup_{ \Gamma_- } |\xi_1|^{-1} \sigma ( Q ) - \textstyle {\frac12} , \end{equation}
and take $ s > s_- $.  For any $ B_1 \in \Psi^0 ( X ) $ satisfying
$ \WF ( I - B_1 ) \cap \Gamma_- = \emptyset $ there exist $ A , B \in \Psi^0 ( X ) $ 
such that 
\[  \Char ( A ) \cap \Gamma_- = \emptyset , \ \ \WF ( B ) \cap \Gamma_- = 
\emptyset \]
and for 
$ u \in \CIc ( X ) $,
\begin{equation}
\label{eq:sinkest}
\| A u \|_{ H^{-s} } \leq C \| B_1 P u \|_{ H^{-s-1} } + 
C \| B u \|_{H^{-s} } + C \| u \|_{H^{-N}} , \end{equation}
for any $ N $. 
\end{prop}

\noindent
{\bf Remark.} A regularization method -- see \cite[Exercise 34]{res} -- 
shows that \eqref{eq:sinkest} is valid for $ u \in \mathcal D' ( X^\circ )$, 
$ \supp u \subset K $ where $ K \Subset X^\circ $ is a fixed set, and for 
which the right hand side of \eqref{eq:sinkest} is finite.


\begin{proof}[Proof of Proposition \ref{p:rad1}]
The basic idea is to produce an operator $ F_s \in \Psi^{s+\frac12} ( X) $,
elliptic on $ \WF ( A ) $ such that for $ s > s_+ $ and $ u \in \CIc ( X ) $, 
we have
\begin{equation}
\label{eq:Fsu}  \| F_s u \|^2_{H^\frac12}  \leq C \| B_1 P u\|_{H^{s} } \| F_s u \|_{ H^{\frac12}} + C \| B_1 u \|_{ H^{s+ \frac12} }^2 + C \| u\|_{ H^{-N} }^2 . 
\end{equation}
This is achieved by writing, in the notation of \eqref{eq:PQR},
\begin{equation}
\label{eq:negcom} 
\begin{split} & \Im \langle P u , F_s^* F_s u \rangle  = 
\langle {\textstyle \frac i 2 } [ P_0 , F_s^* F_s ] u , u \rangle +
 \Re \langle  Q  u , F_s^* F_s u \rangle ,
\end{split}
\end{equation}
and using the first term on the right hand side to control 
the left hand side of \eqref{eq:Fsu}. We note here that since
$ \WF ( F_s ) \cap \WF ( I - B_1 ) = \emptyset $, then in any 
expression involving $ F_s $ we can replace $ u $ and $ P u $ by
$ B_1 u $ and $ B_1 P u $ respectively by introducing errors
$ \mathcal O ( \| u \|_{ H^{-N} } ) $ for any $ N $. Hence from now
on we will consider estimates with $ u $ only. 

To construct a suitable $ F_s $ 
we take $ \psi_1 \in \CIc ( ( - 2\delta, 2 \delta  ) ; [ 0 , 1] ) $, 
$ \psi_1 ( t ) = 1$, for $ |t| < \delta $, $ t \psi_1'( t) \leq 0 $,
and $ \psi_2 \in \CI ( \RR ) $, $ \psi_2 ( t ) = 0 $ for $ t \leq 1 $, 
$ \psi_2 ( t ) = 1$, $ t \geq 2 $,
and propose
\begin{gather*} 
F_s := \psi_1 ( x_1 ) \psi_1 ( -\Delta_h /{D_{x_1}^2 } ) 
\psi_2 ( D_{x_1} ) D_{x_1}^{s+\frac12} \in \Psi^{s+\frac12} ( X ) ,  \\ 
\sigma ( F_s ) =: f_s ( x, \xi ) = \psi_1 ( x_1 ) \psi_1 
( q(x, \xi' ) / \xi_1^2 ) 
\psi_2 ( \xi_1 ) \xi_1^{s+ \frac12} . 
\end{gather*}
We note that because of the cut-off $ \psi_2 $, $ D_{x_1}^{s+\frac12} $
and $ - \Delta_h/D_{x_1}^2 $ are well defined. 

For $ |\xi|  $ large enough (which implies that $ \xi_1 > |\xi|/C  
$ on the support of
$ f_s $ if $ \delta $ is small enough) we use \eqref{eq:Hpmod} to obtain
\begin{equation}
\label{eq:fHfs} \begin{split}  H_p f_s  ( x, \xi ) & = \xi_1^{s +\frac32} \left( 2 x_1 \psi_1'(  x_1 ) 
\psi_1 ( \xi_2/\xi_1 ) + 2 \psi_1 ( x_1 ) 
( q ( x, \xi') / \xi_1^2 ) \psi_1 '(q (x,\xi') /\xi_1^2 ) \right. \\ & \ \ \ \ \ \ \left. - (s +  
{\textstyle \frac12}) \psi_1 ( x_1 ) \psi_1 ( q ( x, \xi') /\xi_1^2 ) \right) \psi_2 ( \xi_1 )  
 \leq - (s + {\textstyle \frac12} ) \xi_1 f_s . \end{split} \end{equation}
In particular,
\begin{equation}
\label{eq:fsign}  f_s H_p f_s + (s + {\textstyle \frac12} ) \xi_1 f_s^2 \leq 0 , \ \ \ \ 
| \xi | > C_0 . 
\end{equation}

The inequality \eqref{eq:fsign} 
is important since $ \sigma ( \frac i 2 [ P_0 , F_s^* F_s ] ) 
= f_s H_p f_s $. Hence returning to \eqref{eq:negcom}, using \eqref{eq:fsign}, 
the sharp G{\aa}rding inequality \cite[Theorem 18.1.14]{ho3} 
and the fact that $ F_s^* [ Q , F_s ] \in \Psi^{ 2 s + 1 } ( X ) $, 
we see that
\[ \begin{split}
\Im \langle P u , F_s^* F_s \rangle & = 
\langle {\textstyle \frac i 2 } [ P_0 , F_s^* F_s ] u , u \rangle +
\langle Q F_s u , F_s u \rangle + \langle F_s^*  [ Q , F_s ] u , u \rangle 
\\ & \leq \langle {\textstyle \frac i 2 } [ P_0 , F_s^* F_s ] u , u \rangle +
\langle Q F_s u , F_s u \rangle + C \| u \|_{ H^{s + \frac12} }^2 \\
& \leq \langle ( - (s + {\textstyle \frac12} ) D_{x_1} + Q ) F_s u , F_s u \rangle
+ C \| u \|_{ H^{s+\frac12}}^2 . \end{split} \]
Since $ D_{x_1 } $ is elliptic (and positive) on $ \WF ( F_s ) $ we can use 
\eqref{eq:WFPQ} to see that if $ s > s_+ $ (where $ s_+ $ is given in 
\eqref{eq:s01}) then 
\[  \begin{split} \| F_s u \|_{ H^{\frac12} }^2 & \leq  
- \Im \langle P u, F_s^* F_s u \rangle + C \| u \|_{ H^{s+\frac12} }^2 
\leq \| P u \|_{H_s} \| F_s^* F_s u \|_{H^{-s}} + C \| u \|_{ H^{s+\frac12 }}^2 \\
& \leq 2 \| P u \|_{H_s }^2 + \textstyle {\frac12} \| F_s u \|_{ H^{\frac12}}^2  + 
C \| u \|_{ H^{s + \frac{1}{2} } } ^2 . \end{split} \]
Recalling the remark made after \eqref{eq:negcom} this gives 
\eqref{eq:Fsu}.  Choosing $ A $ so that $ F_s \in \Psi^{ s + \frac12} $ is elliptic
on $ \WF ( A ) $ we obtain 
\begin{equation}
\label{eq:sourcest1}    \| A u \|_{ H^{s+1} } \leq C \| B_1 P u \|_{ H^s } + C \| B_1 u \|_{ H^{s+\frac12} } 
 C\| u \|_{H^{-N}} .\end{equation}
It remains to eliminate the second term on the right hand side. 
We note that $ \WF ( B_1 ) \cap \Char ( A ) $ forward controlled by 
$ \complement \Char ( A ) $  in the sense of \eqref{eq:control4P}. 
Since \eqref{eq:condp} is satisfied on $ \WF ( B_1 ) \cap \Char ( A )  $ we
apply \eqref{eq:DH} to obtain
\begin{equation}
\label{eq:sourcest2} 
\begin{split}  
\| B_1 u\|_{ H^{s +\frac12 } }  & \leq C 
\| B_2 P u \|_{ H^{s- \frac12 } } + 
C \| A u \|_{ H^{ s+ \frac12 } } + C \| u \|_{H^{-N}}    \\
& \leq C \| B_2 P u \|_{H^{s} } +
\textstyle \frac12  \| A u \|_{ H^{s  }} + C ' \| u \|_{H^{-N}} , 
 \ \ s + \frac12 > - N , 
\end{split}
\end{equation}
where $ B_2 $ has the same propeties as $ B_1 $ but a larger microsupport.
(Here we used an interpolation estimate for Sobolev spaces based on 
$ t^{s + \frac12} \leq \gamma t^{s} + 
\gamma^{-2N-2s -1} t^{-N} $, $ t \geq 0 $ -- that follows from 
rescaling $ \tau^{ s + \frac12} \leq \tau^{s } + \tau^{-N} $, $ \tau \geq 0 $.)

Combining \eqref{eq:sourcest1} and \eqref{eq:sourcest2} gives \eqref{eq:sourcest}
with $ B_1 $ replaced by $ B_2 $. Relabeling the operators concludes the proof.
\end{proof}

\begin{proof}[Proof of Proposition \ref{p:rad2}]
The proof of \eqref{eq:sinkest} is similar to the proof of Proposition 
\ref{p:rad1}.  We now use $ G_s \in \Psi^{ - s -\frac12} ( X ) $ given by the
same formula:
\begin{gather*} G_s := \psi_1 ( x_1 ) \psi_1 ( -\Delta_h /{D_{x_1}^2 } ) 
\psi_2 ( D_{x_1} ) D_{x_1}^{-s-\frac12} \in \Psi^{-s-\frac12} ( X ) ,  \\ 
\sigma ( G_s ) =: g_s ( x, \xi ) = \psi_1 ( x_1 ) \psi_1 
( q(x, \xi' ) / \xi_1^2 ) 
\psi_2 ( \xi_1 ) \xi_1^{-s - \frac12}  .\end{gather*}
However now,
\[ \begin{split} g_s H_g g_s ( x , \xi ) 
& = \xi_1^{-s + \frac12 } g_s ( x, \xi ) \left( 2 x_1 \psi_1'(  x_1 ) 
\psi_1 ( \xi_2/\xi_1 ) + 2 \psi_1 ( x_1 ) 
( q ( x, \xi') / \xi_1^2 ) \psi_1 '(q (x,\xi') /\xi_1^2 ) \right. \\ & \ \ \ \ \ \ \left. - (s +  
{\textstyle \frac12}) \psi_1 ( x_1 ) \psi_1 ( q ( x, \xi') /\xi_1^2 ) \right) \psi_2 ( \xi_1 )  \\
&  \leq - (s + {\textstyle \frac12} ) |\xi_1| g_s^2 + C_0 |\xi_1|^{-2s} b ( x, \xi )^2  , \end{split} \]
 where $ b = \sigma ( B ) $ is chosen to control the terms involving
 $ t \psi_1' ( t ) $ 
 (which now have the ``wrong" sign compared to \eqref{eq:fHfs}).
 The proof now proceeds in the 
 same way as the proof of \eqref{eq:sourcest} but we have to carry over the
 $ \| B u \|_{ H^s} $ terms. 
 \end{proof}

\section{Proof of Theorem 1}
\label{s:t1}

We first show that $ \ker_{ \mathscr X_s } P ( \lambda ) $ is finite
dimensional when $ \Im \lambda > - s - \frac12 $. Using standard arguments
this follows from the definition \eqref{eq:hypXY} and the estimate
\eqref{eq:estker} below. To formulate it 
suppose that 
\[  \chi \in \CIc ( X ) , \ \ \chi|_{ x_1 < - 2\delta} \equiv 0 , \ 
\ \chi |_{ x_1 > - \delta } \equiv 1 , \]
where $ \delta > 0$ is a fixed (small) constant. Then for $ u \in \mathscr X_s $
and $ s > - \Im \lambda - \frac12 $,
\begin{equation}
\label{eq:estker}  \| u \|_{ \bar H^{s+1} ( X^\circ ) } \leq 
C \| P (\lambda ) u \|_{ \bar H^s ( X^\circ ) } + \| \chi u \|_{ H^{-N} ( X) } .
\end{equation}

\begin{proof}[Proof of \eqref{eq:estker}]
If $ \chi_+ \in \CIc $, $ \supp \chi_+ \subset \{ x_1 > 0 \} $
then elliptic estimates show that
\[  \| \chi_+ u \|_{ H^{s+1 } } \leq \| \chi_+ u \|_{ H^{s+2} } \leq 
C \| P u \|_{ H^s } + C \| \chi u \|_{ H^{-N} } . \]
Near $ x_1 = 0 $ we use the estimates 
\eqref{eq:sourcest} (valid for $ u \in \mathscr X_s $) -- see Remark 2 after 
Proposition \ref{p:rad1}) which give for, for $ \chi_0 \in \CIc  $, 
$ \supp \chi_0 \subset \{ |x_1| < \delta /2 \} $
\begin{equation}
\label{eq:Pchi0}   \| \chi_0 u \|_{H^{s +1} ( X )} \leq C \| P ( \lambda )  u \|_{ \bar H^s ( X) } + 
 , \| \chi u \|_{ H^{-N} ( X ) } . \end{equation}
To prove \eqref{eq:Pchi0} we microlocalize to neighbourhoods of $ \{ \pm \xi_1 > |\xi|/C \}$ and use \eqref{eq:sourcest} for $ P ( \lambda ) $ and $ - P ( \lambda ) $ respectively -- 
from \eqref{eq:Plag} we see that $ s_+ = - \Im \lambda - \frac12 $ for $ P= P ( \lambda ) $ and $ s_- = - \Im \lambda - \frac12 $ for $P =  - P ( \lambda ) $ (a rescaling 
by a factor of $ 4 $ is needed by comparing \eqref{eq:Plag} with \eqref{eq:modred}).
Elsewhere the operator is elliptic in $ |x_1|< \delta $. 

Finally if $ \chi_- $ is supported in $ \{ x_1 < -\delta/2 \} $ then 
the hyperbolic estimate \eqref{eq:hypest1} shows that
\[ \| \chi_-  u \|_{ \bar H^{s+1} ( X ) } \leq 
C \| P ( \lambda ) u \|_{ \bar H^s ( X) } + C \| \chi_0 u \|_{ H^{s+1} ( X) }. \]
Putting these estimates together gives \eqref{eq:estker}. \end{proof}

To show that the range of $ P $ on $ \mathscr X_s $ is of finite
codimension and is closed we need the following 

\begin{lem}
\label{l:kerP}
The cokernel of $ P ( \lambda ) $ in $ \dot H^{-s} ( X ) \simeq \mathscr Y_s ^* $
(see \eqref{eq:duality}) 
\[  \coker_{ \mathcal X_s } P ( \lambda ) := \{ v \in \dot H^{-s} ( X ) : \forall 
\, u \in \mathscr X_s,   \ \langle 
P ( \lambda ) u , v \rangle = 0 \} , \]
is equal to the kernel of $ P ( \bar \lambda ) $ on $ \dot H^{-s} ( X )$:
$\coker_{ \mathcal X_s } P ( \lambda ) =  \ker_{ \dot H^{-s} ( X ) } P ( \bar \lambda)$ .
\end{lem}
\begin{proof} 
In view of \eqref{eq:formaladj} 
we have, for $ u \in \bCI ( X^\circ ) $ and $ v \in \dot H^{-s}  ( X ) $, 
\[ \langle P ( \lambda ) u , v \rangle = \langle u , P ( \bar \lambda ) v \rangle .\]
Since $ \bCI ( X^\circ ) $ is dense in $ \mathscr X_s $ (see for instance 
Lemma \cite[Lemma E.42]{res}) it follows that 
$ \langle P ( \lambda ) u , v \rangle = 0 $ for all $ u \in \mathscr X_s $ 
if and only if $ P ( \bar \lambda ) v = 0 $. 
\end{proof} 

Hence to show that  $ \coker_{ \mathscr X_s} $ is finite dimensional 
it suffices to prove that the kernel of $ P ( \bar \lambda ) $ is finite
dimensional. We claim an estimate from which this follows:
\begin{equation}
\label{eq:estcok}   u \in \ker_{ \dot H^{-s} ( X ) } \ \Longrightarrow \ \| u \|_{ \dot H^{-s}  ( X)  } \leq C \| \chi u \|_{ H^{-N} ( X ) } , \ \ 
s > - \Im \lambda - \textstyle{\frac12} \end{equation}
where $ \chi $ is the same as in \eqref{eq:estker}.

\begin{proof}[Proof of \eqref{eq:estcok}]
The hyperbolic estimate \eqref{eq:hypePest} shows that if
$ P (\bar \lambda ) u = 0 $ 
for $ u \in \dot H^{-s} ( X ) $  
(with any $ \lambda \in \CC $ or $ s \in \RR$) 
then  $ u |_{ x_1 < 0 } \equiv 0 $.
We can now apply \eqref{eq:sinkest} with $ P = P ( \lambda ) $ 
near $ \Gamma_- $ and $ P = - P ( \lambda ) $ near $ \Gamma_+ $. 
We again see that the threshold condition is the same at both places:
we require that $ s > - \Im \lambda - \frac12 $. Since $ u $ vanishes
in $ x_1 < 0 $ there $ \WF ( B u ) \cap \Char P ( \lambda ) = \emptyset $ 
and hence (using \eqref{eq:WFPQ}) $ \| B u \|_{ \dot H^{-s} ( X ) } \leq
C \| \chi u \|_{-N} $. Hence \eqref{eq:sinkest} and elliptic estimates
give \eqref{eq:estcok}.
\end{proof}

\section{Asymptotic expansions} 
\label{asym} 

To prove Theorem \ref{t:3} we need a regularity result for $ L^2 $ solutions
of 
\begin{equation}
\label{eq:L2R}  ( - \Delta_g - \lambda^2 - ({\textstyle{ \frac n 2 }})^2)^{-1} u = 
f \in \CIc ( M ) , \ \
\Im \lambda > {\textstyle{\frac n 2 }} . \end{equation}
To formulate it we recall the definition of 
$ X $ given in \eqref{eq:coordX} and of $ X_1 := X \cap \{ x_1 > 0 \}$. 
We also define $ j: M \to X_1 $ to be the natural identification, given by 
$ j ( y_1 , y' ) = (y_1^2 , y') $ near the boundary. Then we have 

\begin{prop}
\label{p:propu}
For $ \Im \lambda \gg 1 $ and $ \lambda \notin i \NN $, the unique $ L^2$-solution 
$ u $ to \eqref{eq:L2R} satisfies
\begin{equation}
\label{eq:propu}
u = y_1^{-i \lambda + \frac n 2} j^* U ,  \ \ U \in \bCI ( X_1 ) . 
\end{equation}
In other words, near the boundary, $ u ( y ) = y_1^{ - i \lambda + \frac n2} 
U ( y_1^2, y') $ where $ U $ is smoothly extendible.
\end{prop}

\noindent
{\bf Remark.} Once Theorem \ref{t:3} is established 
then the relation between $ P ( \lambda )^{-1} $ and the meromorphically 
continued resolvent $ R ( \frac n 2 - i \lambda ) $ shows that 
$ y_1^{ - s } R ( s ) : \dCI ( M ) \to  j^*\bCI ( X_1  ) $ is meromorphic
away from $ s \in \NN $ -- see \S \ref{merc}. That means that away from 
exceptional points \eqref{eq:propu} remains valid for $ u = R ( \frac n 2 - i \lambda ) $. 

To give a direct proof of Proposition \ref{p:propu} we need a few lemmas.
For that we define Sobolev spaces $ H^k_g ( M , d \vol_g ) $ associated to the Laplacian $ - \Delta_g $:
\begin{equation}
\label{eq:Hk}
 H^k_g ( M ) := 
\{ u :  y_1^{|\alpha|} D_{y}^\alpha u  \in 
L^2 ( M , d \! \vol_g ), \  |\alpha| \leq k \}, \ \  \ell \in \NN . 
\end{equation}
(In invariant formulation can 
be obtained by taking vector fields vanishing at $ \partial M $ -- see \cite{mm}.)
Let us also put
\begin{equation}
\label{eq:Qofla}
Q ( \lambda^2) := 
 - \Delta_g - \lambda^2 - ({\textstyle{ \frac n 2 }})^2 .
\end{equation}

\begin{lem}
\label{l:0}
With $ H_g^k ( M ) $ defined by \eqref{eq:Hk} and $ 
Q( \lambda^2 ) $ by \eqref{eq:Qofla} we have for any $ k \geq 0 $, 
\begin{equation}
\label{eq:0}
Q( \lambda^2)^{-1} :
H_g^k ( M ) \to H_g^{k+2} ( M ) , \ \  \Im \lambda > \textstyle{\frac n 2 } . 
\end{equation}
\end{lem}
\begin{proof}
Using the notation from the proof of \eqref{eq:Deltagg} and
 Lemma \ref{l:ahe} we write
\[ Q ( \lambda^2 ) = 
( y_1 D_{y_1} )^2 + y_1^2 d_h^* d- i ( n  + y_1^2 \gamma ( y_1^2, y') ) y_1
D_{y_1} \]
so that for $ u \in \CIc ( M ) $ supported near $ \partial M $, 
and with the inner products in 
$ L^2_g =  L^2 (  M , d \vol_g ) $,
\[ \langle Q ( \lambda^2 ) u , u \rangle_{L^2_g}  = 
\int_M ( | y_1 D_{y_1}|^2 + y_1^2 | d u|_h ^2 ) d \vol_g . \]
Hence,
$  \| u \|_{ H^1_g } \leq C \| Q ( \lambda^2 ) u \|_{ L^2_g } + C \| u \|_{L^2_g } $.
 Using this and expanding  $ \langle Q ( \lambda ) u , Q ( \lambda ) u \rangle_{L^2_g} 
$ we see that 
\[ \| u \|_{ H^2_g } \leq C \| Q ( \lambda^2 ) u \|_{L^2_g } + C \| u \|_{L^2_g } , \ \ 
u \in \CIc ( M ) . \]
Since $ \CIc ( M ) $ is dense in $ H^2_g ( M ) $ it follows that for 
$ \Im \lambda > \frac n 2 $, $ Q ( \lambda)^2 : L^2_g \to H_g^2 $. 
Commuting  $ y_1 V $, where $ V \in \bCI ( M ; TM ) $, with $ Q ( \lambda^2 )$ 
gives the general estimate, 
\[ \| u \|_{ H^{k+2}_g } \leq C \| Q ( \lambda^2 ) u \|_{H^k_g } + C \| u \|_{L^2_g } , \ \ u \in \CIc ( M) ,
\]
and that gives \eqref{eq:0}.
\end{proof}

\begin{lem}
\label{l:01}
For any $ \alpha > 0 $ there exists $ c ( \alpha ) > 0 $ such that
for $ \Im \lambda > c ( \alpha ) $, 
\begin{equation}
\label{eq:01}
 y_1^{\alpha } Q ( \lambda^2 )^{-1}  y_1^{-\alpha}  : L^2_g  ( M ) \to H^2_g ( M ) . 
\end{equation}
\end{lem}
\begin{proof}
We expand the conjugated operator as follows:
\begin{equation}
\label{eq:QlaK} \begin{split}  y_1^{\alpha} Q ( \lambda^2 ) y_1^{-\alpha} & =
Q ( \lambda^2 + \alpha^2 ) - \alpha ( 2 i y_1 D_{y_1 } - n -  y_1^2 \gamma( y_1^2 , y') ) \\
& = \left( I + K ( \lambda, \alpha )  \right)^{-1} Q ( \lambda^2 + \alpha^2) ,\\
K ( \lambda, \alpha ) & := \alpha ( 2 i y_1 D_{y_1 } - n - y_1^2 \gamma ( y_1^2 , y') )
Q ( \lambda^2 + \alpha^2 )^{-1} 
. \end{split} \end{equation}
The inverse $ Q ( \lambda^2 + \alpha^2 ) $ exists due to the following
bound provided by 
the spectral theorem (since $ \Spec ( - \Delta_g ) \subset [ 0 , \infty ) $) 
and \eqref{eq:0} (with $ k =0 $):
\begin{equation}
\label{eq:Qlambdaest}
\| Q ( \mu^2 )^{-1}  \|_{ L^2_g \to H^k_g } 
 \leq \frac{( 1 + C |\mu|)^{k/2}}
{d(\mu^2 , [- ({\textstyle \frac{n}2})^2 , \infty ) ) } , \ \ k = 0 , 2.
\end{equation}
It follows that for $ \Im \lambda > c ( \alpha ) $, $ I + K ( \lambda, \alpha ) $
in \eqref{eq:QlaK} is invertible on $ L^2_g $. Hence we can invert 
$ y_1^\alpha Q ( \lambda^2 ) y_1^{-\alpha} $ with the mapping property 
given in \eqref{eq:01}.
\end{proof} 

\begin{proof}[Proof of Proposition \ref{p:propu}]
The first step of the proof is a strengthening of Lemma \ref{l:0} for 
solutions of \eqref{eq:L2R}. We claim that if $ u $ solves \eqref{eq:L2R}
and $ u \in L^2_g $ then, near the boundary $ \partial M $, 
\begin{equation}
\label{eq:Qlambda}
V_1 \cdots V_N u \in L^2_g ,  \ \
V_j \in \bCI ( M , T M ) , \ \  V_j y_1 |_{ y_1} = 0 , 
\end{equation}
for any  $N$. The condition on $ V_j $ means that $ V_j $ are tangent to the
boundary $ \partial M $ (for more on spaces defined by such conditions see
\cite[\S 18.3]{ho3}).

To obtain \eqref{eq:Qlambda} 
we see that if $ V $ is a vector field tangent to the boundary of $ \partial M $
then 
\[ \begin{split}  Q ( \lambda^2 ) V u & = F := V f + [ ( y_1 D_{y_1} )^2 , V ] u + 
y_1^2 [ \Delta_{ h(y_1^2 ) } , V ] - i [ ( n +  y_1^2 \gamma ( y ) )  y_1 D_{y_1}  , 
V ] \\
& = Vf + y_1^2 Q_2 u + y_1 Q_1 u  ,
\end{split}  \]
where $ Q_j $ are differential operators of order $ j$. Lemma \ref{l:0} shows
that $ F \in L^2_g $. From Lemma \ref{l:0} we also know that $ y_1 V u \in L^2_g $.
Hence, 
\[  y_1 V u - y_1 Q ( \lambda^2 )^{-1} F \in L^2_g , \ \
Q ( \lambda^2 ) y_1^{-1} (  y_1 V u - y_1 Q ( \lambda^2 )^{-1} F  ) = 0 . \]
But for $ \Im \lambda > c_0 $, Lemma \ref{l:01} shows that
\begin{equation}
\label{eq:injQla} 
 Q ( \lambda^2 ) y_1^{-1} v = 0 , \ \ v \in L^2 (M , d\! \vol_g )
\ \Longrightarrow \ v = 0 . 
\end{equation}
Hence $ V u = Q( \lambda^2 )^{-1} F \in L^2_g $. 
This argument can be iterated showing \eqref{eq:Qlambda}.

We now consider $ P ( \lambda ) $ as an operator on $ X_1 $, formally 
selfadjoint with respect to $ d \mu = dx_1 d \vol_h $. Since we are on 
open manifolds the two $ \CI $ structures agree and we can consider 
$ P ( \lambda ) $ as operator on $ \CI ( M  ) $. 
Since
\[ Q ( \lambda^2)  = y_1^{-i \lambda + \frac n 2}
y_1^2 P ( \lambda ) y_1^{ i \lambda - \frac n 2  }  = x_1^{ -\frac{i\lambda} 2 + 
\frac n 4 } x_1 P ( \lambda ) x_1^{ \frac{ i \lambda }2 - \frac n4}
, \]
we can define
\begin{equation}
\label{eq:Tla} T ( \lambda ) := x_1^{ \frac{i \lambda }2 - \frac n 4 } 
Q ( \lambda^2 )^{-1} 
 x_1^{ -\frac{i\lambda}2 + \frac{n}4 + 1 } , \ \ \Im \lambda > 
 {\textstyle \frac  n 2}, \end{equation}
which satisfies
\begin{gather}
\label{eq:Tla1}
\begin{gathered}   P ( \lambda ) T ( \lambda ) f = f , \ \ f \in \CIc ( X_1 ) , \\ 
T ( \lambda ) :  x_1^{ - \frac{\rho}2 - \frac12 } 
 L^2   \to 
 x_1^{ - \frac{\rho}2 + \frac12 } L^2 , 
 \ \ \rho := \Im \lambda > {\textstyle \frac  n 2} . 
\end{gathered}
\end{gather}
Here we used the fact that $ 2dy_1/y_1 = dx_1/x_1 $
and that 
\[ L^2 ( y^{-n-1}_1 dy_1 d \! \vol_h ) = 
 L^2 \left( x_1^{-\frac n 2 - 1}  {dx_1 d \! \vol_h }\right) = 
x_1^{\frac n 4 + \frac12  } L^2 , \ \  L^2 :=  L^2 ( dx_1 d \!\vol_h ) . \]
Proposition \ref{p:propu} is equivalent to the following 
mapping property of $ T ( \lambda ) $:
\begin{equation}
\label{eq:P0las}
T ( \lambda )  : \CIc ( X_1 ) \longrightarrow 
\bCI ( X_1 ) , \ \ \ \Im \lambda \geq c_0 , \ \ \lambda \notin 
i \NN .  
\end{equation}
To prove \eqref{eq:P0las} we will use a classical tool 
for obtaining asymptotic expansions, the {\em Mellin transform}.
Thus let $ u = T ( \lambda ) f $, $ f \in \CIc (X_1) $. 
By replacing $ u $ by $ \chi ( x_1 ) u $, $ \chi \in \CIc ( (-1,1) ;[0,1])$,  $ \chi = 1 $  near $ 0 $, we can assume that 
\[ u \in \CI ( ( 0,1 ) \times \partial M ) \cap x_1^{-\frac{\rho}2 + \frac12 } L^2 , \ \  P(\lambda ) u = f_1 \in \CIc ( ( 0 , 1 ) \times \partial M ) , \ \ \rho > {\textstyle \frac  n 2} , \]
where smoothness for $ x_1 > 0 $ follows from Lemma \ref{l:0}. In addition 
\eqref{eq:Qlambda} shows that 
\begin{equation}
\label{eq:Qlambdax}
V_1 \cdots V_N u \in x_1^{-\frac \rho 2 + \frac12 } L^2 ( dx_1 d\! \vol_h ) ,  \ \
V_j \in \bCI ( X_1 , T X_1 ) , \ \  V_j x_1 |_{ x_1} = 0 .
\end{equation}
In particular, for any $ k $ 
\begin{equation}
\label{eq:x1N}  x_1^N u \in C^k ( [0,1]\times \SP^1 ) 
\end{equation}
if $ N $ is large enough. 

We define the Mellin transform (for functions with support
in $ [0,1)$) as 
\[  M u ( s, x' ) :=  \int_0^1 u ( x)  x_1^{ s  } \frac {dx_1}{x_1}  .   \]
This is well defined for $ \Re s >  \rho /2  $:
\[  \begin{split} \| M u ( s , x') \|_{ L^2 ( d \! \vol_h ) }^2  
& = \int_{\SP^1} \left|  \int_0^1 x_1^{  s +\frac{ i \lambda}2 - \frac 1 2} 
( x_1^{ - \frac{  i \lambda} 2 - \frac 1 2 } u ( x_1, x' ) ) dx_1 
\right|^2 d \! \vol_h  \\
& \leq \left( \int_0^1 t^{ - \rho + 2 \Re s - 1 } dt \right)
\| x_1^{  \frac \rho 2 - \frac 1 2 } u \|_{ L^2 } 
= ( 2\! \Re \!s \!- \!\rho)^{-1} \| x_1^{ \frac \rho 2 - \frac 1 2 } u \|_{ L^2 } . 
\end{split} \]
In view of \eqref{eq:Qlambda}  
$ s \longmapsto  M u ( s , x_2 ) $ 
is a holomorphic family of {\em smooth} functions in 
$ \Re s > \rho /2  $. We claim now that $ M u ( s, x' ) 
$ continues meromorphically to all of $ \CC $. In fact, from \eqref{eq:Plag} 
we see that for $ f_2 := \frac14 f_1 $, 
\[  M (x_1 f_2) ( s , x') = M ( {\textstyle{\frac 14}} x_1 P ( \lambda ) u ) ( s , x' ) = 
- s ( s + i \lambda ) M u ( s, x' ) + M ( Q_2 u)  ( s+1, x' ) , \]
where $ Q_2 $ is a second order differential operator built out of
vector fields tangent to the boundary of $ X_1 $. In view of 
\eqref{eq:Qlambdax} $ Q_2 u \in x_1^{- \frac \rho 2 + \frac 12} L^2 $.
Also, $  s \mapsto M ( x_1 f_2 ) ( s, x' ) $ is entire as $ f_1 $
vanishes near $ x_1 = 0 $. 
Hence, 
\[ \begin{split}  M u ( s , x' ) = &  
 \frac{ M ( Q_2^k u ) ( s + k + 1 , x' ) } { s ( s + i \lambda ) \cdots ( s + k ) ( s + k + i \lambda ) } 
- \sum_{ j=0}^k 
\frac{  M Q_1^{j}( x_1 f_2 ) ( s + j , x' ) }{ 
s ( s + i \lambda ) \cdots ( s + j ) ( s + j + i \lambda ) } ,
\end{split} \]
and that provides a meromorphic continuations with possible 
poles at $ - i \lambda - k $, $ k \in \NN $. 

The Mellin transform inversion formula, a contour deformation
and the residue theorem (applied to simple poles  thanks 
to our assumption that $ i \lambda \notin \ZZ $) then give
\[  u ( x ) \simeq x_1^{i\lambda } ( b_0 (x' ) + x_1 b_1 ( x' ) + 
\cdots ) + a_0  (x' ) + x_1 a_1 ( x' ) + \cdots , \ \ a_j , b_j \in \CI ( \partial M ),  \]
where the regularity of remainders comes from \eqref{eq:x1N}. 
(The basic point is that 
$$ M ( x_1^a \chi ( x_1 ) ) ( s ) = ( s + a )^{-1} F ( s ) , \ \ 
 F ( s ) = - \int x_1^{ a + s }\chi' ( x_1 ) dx_1 , $$ so that
 $ F ( s ) $ is an entire function with $ F ( -a ) = 1 $.)

Since  $P  u ( x ) = 0 $ for $ 0 < x_1 < \epsilon $ the equation
shows that $ b_k $ is determined by $ b_0, \cdots b_{k-1} $. We claim that 
$ b_k \equiv 0 $: if $ b_0 \neq 0 $ then 
\[  |x_1^{ \frac{ \rho} 2 - \frac12} u | = x_1^{- \frac 1 2} | b_0 ( x') | + 
\mathcal O ( x_1^{\frac12} ) \notin L^2 ( dx_1 d\! \vol_h ) . \]
contradicting \eqref{eq:Qlambdax}. It follows that $ u \in \bCI (X_1 ) $
proving \eqref{eq:P0las} and completing the proof of Proposition \ref{p:propu}.
\end{proof}

\section{Meromorphic continuation}
\label{merc}

To prove Theorem \ref{t:3} we recall that $ ( - \Delta_g - \lambda^2 - (\frac n2 )^2 )^{-1} $ is a 
holomorphic family of operators on $ L^2_g $ for $ \lambda^2 + ( \frac n 2 )^2 \notin
\Spec ( - \Delta_g ) $ and in particular for $ \Im \lambda > \frac n 2 $.
 
\begin{proof}[Proof of Theorem \ref{t:3}]
We first show that for 
$ \Im \lambda > 0 $, $  \lambda^2 + {\textstyle \frac14} \notin \Spec(-\Delta_g)$, 
\begin{equation}
\label{eq:Pla0}  P ( \lambda ) u = 0 , \ \ u \in \mathscr X_s , \ \ 
s > - \Im \lambda - \textstyle{\frac12}  \ \Longrightarrow  \ u \equiv 0 .
\end{equation}
In fact, from \eqref{eq:PtoC} we see that $ u \in \bCI ( X ) $. Then 
putting $ v ( y ) := y_1^{ - i \lambda + \frac n 2} j^* (u|_{X_1} )  $, 
$ j: M \to X_1 $,  \eqref{eq:Plag} shows that $ ( - \Delta_g - \lambda^2 - 
(\frac n 2 )^2 ) v = 0 $. For $ \Im \lambda> 0 $ we have $ v \in L^2_g $
and hence from our assumptions, $ v \equiv 0 $. 
Hence $ u |_{X_1 } \equiv 0 $,  and $  u \in \bCI ( X ) $. Lemma \ref{l:ahe}
then shows that $ u \equiv 0 $ proving \eqref{eq:Pla0}.

In view of Lemma \ref{l:kerP} 
we now need to show that $ P ( \lambda )^* w = 0   $,
$ w \in \dot H^{-s} ( X) $, implies that $ w \equiv 0 $. It is 
enough to do this for $ \lambda_0 \notin i \NN $ and $ \Im \lambda \gg 1 $
since invertibility at one point shows that the index of $ P ( \lambda ) $ is $ 0 $.
Then \eqref{eq:Pla0} shows invertibility for all $ \Im \lambda > 0 $, 
$ \lambda^2 + (\frac n 2)^2 \in \Spec ( - \Delta_g ) $. 

Hence suppose that $ P ( \lambda )^* w = 0 $, $ w \in \dot H^{-s} ( X ) $.
Estimate \eqref{eq:hypePest} then shows that $ \supp w \subset \overline X_1 $.
(For $ -1 < x_1 < 0 $ we solve a hyperbolic equation with zero initial 
data and zero right hand side.) 
We now show that $ \supp w \cap  X_1 \neq \emptyset $ (that is 
there is some support in $ x_1 > 0$; in fact
by unique continuation results for second order elliptic operators,
see for instance \cite[\S 17.2]{ho3}, this shows that 
$ \supp w = \overline X_1 $). In other words we 
we need to show that
we cannot have $ \supp w \subset \{x_1 = 0 \} $. Since 
$ \WF ( w ) \subset N^* \partial X_1 $ we can restrict 
$ w $ to fixed values of $ x' \in \partial M $ and the restriction 
and is then a linear combination of $ \delta^{(k)} ( x_1 ) $. But 
\[ P ( \bar \lambda  ) ( \delta^{(k)} ( x_1 ) ) = 
( k + 1 - \bar \lambda / i ) \delta^{(k+1) } (x_1) - i \gamma ( x ) ( 2 i ( k+1)  
- \bar \lambda - i {\textstyle\frac{n-1} 2 } ) \delta^{(k)} ( x_1 ) , \]
and that does not vanish for $ \Im \lambda  > 0 $. 

Mapping property \eqref{eq:P0las} and the definition of  $P ( \lambda ) $
show that for any $ f \in \CIc ( X_1 ) $ (that is $ f $ supported
in $ x_1 > 0 $) there exists $ u \in \bar C^\infty ( X_1 ) $ such 
that $ P ( \lambda ) u = f $ in $ X_1 $. Then (with $L^2 $ inner products
meant as distributional pairings), 
\[ \langle f , w \rangle = \langle P ( \lambda ) u , w \rangle =
\langle u , P ( \lambda ) ^* w \rangle = 0 . \]
Since $ w \in \dot {\mathcal D} ( X_1 ) $ and $ u \in \bar C^\infty ( X_1 ) $ the pairing is justified. In view of support properties of
$ w $, we can find $ f $ such that the left hand side does not 
vanish. This gives a contradiction. 
\end{proof}

\noindent
{\bf Remark.} 
Different proofs
of the existence of $ \lambda $ with $ P( \lambda ) $ invertible 
can be obtained using semiclassical versions 
of the propagation estimates of \S \ref{radest}. That is done for 
$ \Im \lambda_0 \gg \langle \Re \lambda_0\rangle $ in \cite{vasy2}
and for $ \Im \lambda_0 \gg 1 $ in \cite[\S 5.5.3]{res}. 


Theorem \ref{t:3} guarantees existence of the inverse at many values of $ \lambda$.
Then standard Fredholm analytic theory (see for instance \cite[Theorem C.5]{res})
gives
\begin{equation}
\label{eq:Plainv}
P ( \lambda)^{-1} : \mathscr Y_s \to \mathscr X_s \ \text{ is a 
meromorphic family of operators in $ \Im \lambda > - s -\frac12 $.}
\end{equation}

\begin{proof}[Proof of Theorem \ref{t:1}]
We define
\[ V ( \lambda ) : \CIc ( M ) \to \CIc ( X ) , \ \ \  f ( y ) \longmapsto T f ( x ) 
:= \left\{ \begin{array}{ll} x_1^{ \frac {i  \lambda} 2  - \frac n 4 - 1} (j^{-1})^* f  & x_1 > 0 , \\
\ \ \ \ 0, & x_1 \leq 0 , \end{array} \right. \]
\[ U ( \lambda ) : \bCI ( X ) \to \CI ( M ) , \ \ \ 
u ( x ) \longmapsto y_1^{ - i \lambda + \frac  n 2} j^* ( u|_{X_1} )  , \]
where $ j : M \to X_1 $ is the map defined by $ j ( y) = ( y_1^2, y') $ near
$ \partial M $. Then,  for $ \Im \lambda > \frac n2 $, \eqref{eq:firstconj} 
and \eqref{eq:Plag} show that 
\begin{equation}
\label{eq:Rla}
R ( \textstyle \frac{ n } 2  - i \lambda ) =  U ( \lambda ) P ( \lambda )^{-1} V ( \lambda ) .
\end{equation}
Since $ P ( \lambda )^{-1} : \bCI ( X ) \to \bCI ( X) $ is a meromorphic
family of operators in $ \CC $, Theorem \ref{t:1} follows.
\end{proof}

\noindent
{\bf Remarks.} 
1. The structure of the residue of $ P ( \lambda )^{-1} $ 
is easiest to describe when the pole at $ \lambda_0 $ 
is simple and has rank one. In that
case, 
\begin{gather*}  P ( \lambda ) = \frac{ u \otimes v }{ \lambda - \lambda_0 } + Q ( \lambda, \lambda_0 ) , \ \   u \in \bCI ( X ) , \  \ v \in \!\!\!\!\!\bigcap_{ s > - \Im \lambda_0 - \frac12} 
\dot H^{-s} (\overline X_1 ) \, \\ P ( \lambda_0 ) u = 0 , \ \ P ( \bar \lambda_0 ) v = 0 , \end{gather*}
and where $ Q ( \lambda, \lambda_0 ) $ is holomorphic near $ \lambda_0 $. 
We note that $ u \in \CI ( X) $ because of \eqref{eq:PtoC}. The regularity 
of $ v \in \dot H^{-s}$, $ s >  - \Im \lambda_0 - \frac12 $ just misses the 
threshold for smoothness -- in particular there is no contradiction with 
Theorem \ref{t:3}! 

\noindent
2. The relation \eqref{eq:Rla} between $ R ( \frac n 2 - i  \lambda  ) $ and 
$ P ( \lambda ) $ shows that unless the elements of the kernel of 
$ P ( \bar \lambda ) $ are supported on $ \partial X_1  = \{ x_1 = 0 \} $ then 
the multiplicities of the poles of $ R ( \frac n 2 - i \lambda ) $ agree. 

For completeness we conclude with the proof of the following standard fact:
\begin{prop}
\label{p:hintz}
If $ R ( \zeta ) := ( - \Delta_g - \zeta ( n - \zeta ) )^{-1} $ for 
$ \Re \zeta > n $ then 
\begin{equation}
\label{eq:hintz0}  R ( \zeta ) : L^2 ( M , d \! \vol_g ) \to L^2 ( M , d \! \vol_g ) , \end{equation}
is meromorphic for $ \Re \zeta > \frac n 2 $ with simple poles where
$ \zeta ( n - \zeta ) \in \Spec ( - \Delta_g ) $.
\end{prop}
\begin{proof}
The spectral theorem implies that $ R ( \zeta ) $ is holomorphic on 
$ L^2_g $ in $ \{ \Re \zeta > \frac n 2 \} \setminus  [ \frac n 2  , n ] $. 
In the $ \lambda $-plane that corresponds to $ \{ \Im \lambda > 0 \}
\setminus i [ 0 , \frac n 2 ] $. 

From \eqref{eq:Tla} and \eqref{eq:Tla1} we see that 
boundeness of $ R ( \frac n 2 - i \lambda ) $ on $ L^2_g ( M ) $ is equivalent
to 
\begin{equation}
\label{eq:hintz1} P ( \lambda)^{-1} : x_1^{ -\frac \rho 2 - \frac12} L^2 ( X_1 ) \to 
x_1^{ - \frac \rho 2 + \frac 12} L^2 ( X_1 )  , \ \ \rho := \Im \lambda .
\end{equation}
We will first 
prove \eqref{eq:hintz} for $ 0 < \rho \leq 1$. From Theorem \ref{t:3} we know that
except at a discrete set of poles, $ P ( \lambda )^{-1} : 
\bar H^s ( X_1 ) \to \bar H^{s+1} ( X_1 ) $, $ s > - \rho - \frac12$.
We claim that for $ - 1 \leq s < - \frac12  $ 
\begin{equation}
\label{eq:Sob}
x_1^{s } L^2 ( X_1 ) \hookrightarrow \bar H^s ( X_1 ) , \ \ 
\bar H^{s+1} ( X_1 ) \hookrightarrow x_1^{s+1} L^2 ( X_1 ) .
\end{equation}
By duality the first inclusion follows from the inclusion
\begin{equation}
\label{eq:hintz} 
\dot H^{r} ( X_1 ) \hookrightarrow x_1^{r} L^2 , \ \ 0 \leq r \leq 1 .
\end{equation}
Because of interpolation we only need to prove this for $ r = 1$ in which
case it follows from Hardy's inequality, 
$\int_0^{\infty } | x_1^{-1} u ( x_1 ) |^2 dx_1 \leq 4 \int_0^\infty 
| \partial_{x_1}^2 u ( x_1 ) |^2 dx_1 $. The second inclusion 
follows from \eqref{eq:hintz} and the fact that 
$ \bar H^r ( X_1 ) = \dot H^r ( X_1) $ for $ 0 \leq r < \frac12 $ -- see
\cite[Chapter 4, (5.16)]{Ta}. 
We can now take $ s = - \frac \rho 2 - \frac 12 $ in \eqref{eq:Sob} 
which for $ 0 < \rho \leq 1 $ is in the allowed range. That proves
\eqref{eq:hintz1} for $ 0 < \Im \lambda  \leq 1$, except at the poles
and consequently establishes \eqref{eq:hintz0} for $ \frac n 2 < \Re s \leq
\frac n 2 + 1 $. We can choose a polynomial $ p ( s ) $ such
that $ p ( s ) R ( s ) : \CIc ( M ) \to \CI ( M ) $ is holomorphic 
near $ [ \frac n2 , n ] $. The maximum principle applied to 
$ \langle p ( s ) R ( s ) f, g \rangle$, $ f, g \in \CIc ( M ) $ 
now proves that $ p ( s ) R ( s )$ is bounded on $ L^2_g ( M ) $ near 
$ [ \frac n 2 , n ] $ concluding the proof.
\end{proof}

\def\arXiv#1{\href{http://arxiv.org/abs/#1}{arXiv:#1}}

\end{document}